\documentclass[11pt,a4paper]{amsart}
\usepackage{amsmath,amssymb, amsbsy}
\usepackage{subfigure}
\usepackage{graphpap,latexsym,epsf}
\usepackage{color,psfrag}
\usepackage[dvips]{graphicx}
\usepackage{enumerate}
\usepackage{bbm}
\usepackage{relsize}
\begin{document}
\newcommand{\dyle}{\displaystyle}
\newcommand{\R}{{\mathbb{R}}}
\newcommand{\Hi}{{\mathbb H}}
\newcommand{\Ss}{{\mathbb S}}
\newcommand{\N}{{\mathbb N}}
\newcommand{\Rn}{{\mathbb{R}^n}}
\newcommand{\F}{{\mathcal F}}
\newcommand{\ieq}{\begin{equation}}
\newcommand{\eeq}{\end{equation}}
\newcommand{\ieqa}{\begin{eqnarray}}
\newcommand{\eeqa}{\end{eqnarray}}
\newcommand{\ieqas}{\begin{eqnarray*}}
\newcommand{\eeqas}{\end{eqnarray*}}
\newcommand{\f}{\hat{f}}
\newcommand{\Bo}{\put(260,0){\rule{2mm}{2mm}}\\}
\newcommand{\1}{\mathlarger{\mathlarger{\mathbbm{1}}}}


\theoremstyle{plain}
\newtheorem{theorem}{Theorem} [section]
\newtheorem{corollary}[theorem]{Corollary}
\newtheorem{lemma}[theorem]{Lemma}
\newtheorem{proposition}[theorem]{Proposition}
\def\neweq#1{\begin{equation}\label{#1}}
\def\endeq{\end{equation}}
\def\eq#1{(\ref{#1})}


\theoremstyle{definition}
\newtheorem{definition}[theorem]{Definition}
\newtheorem{remark}[theorem]{Remark}
\numberwithin{figure}{section}

\title[Fluctuation-dissipation and stochastic calculus]{Fluctuation-dissipation relation,
Maxwell-Boltzmann statistics, equipartition theorem, and stochastic calculus}

\author[C. Escudero]{Carlos Escudero}
\address{}
\email{}


\date{\today}

\begin{abstract}
We derive the fluctuation-dissipation relation and explore its connection with the equipartition theorem and
Maxwell-Boltzmann statistics through the use
of different stochastic analytical techniques. Our first approach is the theory of backward stochastic differential equations,
which arises naturally in this context, and facilitates the understanding of the interplay between these classical results
of statistical mechanics. The second approach consists in deriving forward stochastic differential equations for the energy of
an electric system according to both It\^o and Stratonovich stochastic calculus rules. While the It\^o equation possesses a unique solution,
which is the physically relevant one, the Stratonovich equation admits this solution along with infinitely many more, none of which has
a physical nature. Despite of this fact, some, but not all of them, obey the fluctuation-dissipation relation.
\end{abstract}
\maketitle

\section{Introduction}

The classical papers by Einstein~\cite{einstein} and Langevin~\cite{langevin}
anticipated two significant tools that have been extensively employed in the study
of stochastic processes relevant in the field of statistical physics along the twentieth century. These tools are, of course, partial
differential equation methods~\cite{einstein} and stochastic analytical methods~\cite{langevin,uo}. In a sense, partial differential
equation methods have been more popular in the study of some aspects of statistical physics~\cite{risken}. Nevertheless, stochastic methods yield
additional insights into statistical mechanical systems that could be, at times, even more enlightening~\cite{gillespie}.
The aim of this work is to approach the classical fluctuation-dissipation relation with stochastic methods that yield additional information to
the classical ones~\cite{gillespie,langevin}.

In section~\ref{bsdes} we approach the fluctuation-dissipation relation through
the Langevin model for the random dispersal of a Brownian particle.
According to this model the position $X_t$ of such a particle obeys Newton second law~\cite{fgl}
\begin{eqnarray}\label{nsl}
m \, \frac{d^2 X_t}{dt^2} &=& -\gamma \, \frac{d X_t}{dt} + \sigma \, \xi_t, \\ \nonumber
\left. X_t \right|_{t=0} &=& X_0, \\ \nonumber
\left. \frac{d X_t}{dt} \right|_{t=0} &=& V_0,
\end{eqnarray}
where $\xi_t$ is Gaussian white noise and $m, \gamma, \sigma >0$ denote respectively the mass of the particle,
the viscosity of the medium in which the particle is immersed, and the strength of the thermal fluctuations.
The random variables $X_0$ and $V_0$ denote, obviously, the initial position and velocity of the particle.
The traditional approach focuses on analyzing the long time behavior of
the forward stochastic differential equation for the velocity of this particle
\begin{equation}\label{velocity}
V_t = \frac{d X_t}{dt};
\end{equation}
and the fluctuation-dissipation relation arises as a consequence of Maxwell-Boltzmann statistics
in this limit. Therefore it is natural to study this problem posed backwards in time, with these
statistics imposed on the final condition. This is the program carried out in section~\ref{bsdes}.

The fluctuation-dissipation relation does not only manifest itself in the random dispersal of a particle subjected to
a heat bath and embedded in a viscous medium. Another paradigmatic example of this relation arises in electric circuits,
in which the stochasticity enters through the fluctuations in the electric current known as Johnson noise~\cite{gillespie,johnson,nyquist}.
In this system, the fluctuation-dissipation relation appears as a consequence of the equipartition of energy in the long time limit.
While building a theory \emph{\`a la Langevin} for this phenomenon requires considering a stochastic differential equation for the
electric current~\cite{gillespie}, it is natural to directly consider the equation for the energy instead.
This is the approach developed in section~\ref{fsdes}, where we explore the technical difficulties that such a method implies. In particular,
the equation for the current presents additive noise, and the equation for the energy shows multiplicative noise, and therefore it should be
interpreted. While the It\^o interpretation yields no problems, the Stratonovich interpretation is affected by an infinite multiplicity
of solutions, of which only one has physical nature. We analyze how some of these spurious solutions still obey the fluctuation-dissipation
relation, while others do not. The link of the results in this section
with one of the most classical versions of the It\^o vs Stratonovich dilemma is shown in the following section~\ref{ideal}.

Finally, in section~\ref{conclusions}, we draw our main conclusions. On one hand, we discuss the possible interest of the present
stochastic methods to address these physical problems. On the other hand, we comment on how our results could fit in the physics literature.
All in all, we hope we can offer new viewpoints on such a classical and important result as the fluctuation-dissipation relation.

\section{A backward stochastic differential equations approach to the fluctuation-dissipation relation}
\label{bsdes}

Let $(\Omega,\mathcal{F},\{\mathcal{F}_t\}_{t \ge 0},\mathbb{P})$ be a filtered probability space
completed with the $\mathbb{P}-$null sets in which a Wiener Process $\{W_t\}_{t \ge 0}$ is defined; moreover assume
$\mathcal{F}_t \supset \sigma(\{W_s, 0 \le s \le t\})$.
Throughout this section we consider the random dispersal of a Brownian particle that will be assumed to be of
unit mass, that is $m=1$, without loss of generality.
We consider the backward stochastic differential equation
\begin{eqnarray}\label{bl}
dV_t &=& -\gamma V_t dt + \sigma_t dW_t, \\ \nonumber
V_T &=& F,
\end{eqnarray}
where $0 \le t \le T$ and $F$ is $\mathcal{F}_T-$measurable.
Note that this is just a final value problem for the velocity of the particle, defined by equation~\eqref{velocity},
which position is described by equation~\eqref{nsl},
but with the amplitude of the fluctuations promoted from constant to stochastic process.
In other words, our plan is to solve the Langevin equation backwards in time for both the velocity $V_t$ and the diffusion $\sigma_t$.
Solving a single backward stochastic differential equation for two stochastic processes is key for keeping the adaptability of the solutions~\cite{pp}.

Furthermore assume $F \in L^2(\Omega)$.
Following the theory developed in~\cite{pp} we know equation~\eqref{bl} possesses a unique solution;
moreover the stochastic process $V_t$, that is, the velocity of the Brownian particle,
admits the following explicit representation:
$$
V_t= e^{\gamma (T-t)} \mathbb{E} [F|\mathcal{F}_t].
$$
Clearly, the fluctuation-dissipation relation can only be established through the explicit computation
of the diffusion $\sigma_t$. We complete this calculation using Malliavin calculus and proceed in two steps;
first we make the change of variables
$$
U_t= e^{\gamma (t-T)} V_t,
$$
and therefore
\begin{eqnarray}\nonumber
dU_t &=& \gamma e^{\gamma (t-T)} V_t + e^{\gamma (t-T)} dV_t \\ \nonumber
&=& \gamma e^{\gamma (t-T)} V_t - \gamma e^{\gamma (t-T)} V_t + e^{\gamma (t-T)} \sigma_t dW_t,
\end{eqnarray}
thus
$$
dU_t = e^{\gamma (t-T)} \sigma_t dW_t,
$$
or equivalently
\begin{eqnarray}\nonumber
&& U_t = U_0 + \int_{0}^{t} e^{\gamma (s-T)} \sigma_s dW_s \\ \nonumber
&\Longrightarrow& e^{\gamma (t-T)} V_t = e^{-\gamma T} V_0 + \int_{0}^{t} e^{\gamma (s-T)} \sigma_s dW_s \\ \nonumber
&\Longrightarrow& V_t = e^{-\gamma t} V_0 + \int_{0}^{t} e^{\gamma (s-t)} \sigma_s dW_s \\ \nonumber
&\Longrightarrow& V_t = \mathbb{E}[V_t] + \int_{0}^{t} e^{\gamma (s-t)} \sigma_s dW_s \\ \nonumber
&\Longrightarrow& V_T = \mathbb{E}[V_T] + \int_{0}^{T} e^{\gamma (s-T)} \sigma_s dW_s \\ \nonumber
&\Longrightarrow& F = \mathbb{E}[F] + \int_{0}^{T} e^{\gamma (s-T)} \sigma_s dW_s,
\end{eqnarray}
by the zero mean property of the It\^o integral. Our second step is to use the generalization
of the Clark-Ocone formula for $L^2(\Omega)$ random variables~\cite{noep} to find
$$
e^{\gamma (s-T)} \sigma_s = \mathbb{E}[D_s F | \mathcal{F}_s],
$$
where $D_s F$ is the Malliavin derivative of $F$, so we may conclude
$$
\sigma_t = e^{\gamma (T-t)} \mathbb{E}[D_t F | \mathcal{F}_t].
$$
Summarizing, we have found the unique solution pair to the backward stochastic differential equation
\begin{subequations}
\begin{eqnarray}\label{v1}
V_t &=& e^{\gamma (T-t)} \mathbb{E} [F|\mathcal{F}_t], \\ \label{s1}
\sigma_t &=& e^{\gamma (T-t)} \mathbb{E}[D_t F | \mathcal{F}_t].
\end{eqnarray}
\end{subequations}
To recover the classical results we need to assume that the velocity is Maxwell-Boltzmann distributed at the terminal time, that is
$$
F \sim \mathcal{N}(0,k_B \tau),
$$
where $k_B$ is the Boltzmann constant and $\tau>0$ is the absolute temperature.
Furthermore assume the representation of this random variable by means of an It\^o integral
$$
F= \sqrt{k_B \tau} \frac{\int_{0}^{T} \psi_s dW_s}{\left(\int_{0}^{T} \psi_s^2 ds \right)^{1/2}},
$$
where $\psi_s \in L^2(0,T)$ is an arbitrary deterministic function,
to find
\begin{subequations}
\begin{eqnarray}\label{v2}
V_t &=& \sqrt{k_B \tau} e^{\gamma (T-t)} \frac{\int_{0}^{t} \psi_s dW_s}{\left(\int_{0}^{T} \psi_s^2 ds \right)^{1/2}}, \\ \label{s2}
\sigma_t &=& \sqrt{k_B \tau} e^{\gamma (T-t)} \frac{\psi_t}{\left(\int_{0}^{T} \psi_s^2 ds \right)^{1/2}}.
\end{eqnarray}
\end{subequations}
As these formulas reveal, the classical fluctuation-dissipation relation is not recovered
by simply assuming an approach to equilibrium driven by linear viscous damping and
Maxwell-Boltzmann statistics once there.
To recover it let us moreover assume the Ornstein-Uhlenbeck form
$$
\psi_t \propto e^{\gamma (t-T)}
$$
to obtain
\begin{subequations}
\begin{eqnarray}\label{v3}
V_t &=& \sqrt{2 \gamma k_B \tau} e^{\gamma (T-t)} \frac{\int_{0}^{t} e^{\gamma (s-T)} dW_s}{\left( 1- e^{-2 \gamma T} \right)^{1/2}}, \\ \label{s3}
\sigma_t &=& \sqrt{2 \gamma k_B \tau} e^{\gamma (T-t)} \frac{e^{\gamma (t-T)}}{\left( 1- e^{-2 \gamma T} \right)^{1/2}}.
\end{eqnarray}
\end{subequations}
Then we finally conclude
\begin{eqnarray}\nonumber
V_t &=& \sqrt{2 \gamma k_B \tau} e^{-\gamma t} \frac{\int_{0}^{t} e^{\gamma s} dW_s}{\left( 1- e^{-2 \gamma T} \right)^{1/2}} \\ \nonumber
&\underset{T \nearrow \infty}{\longrightarrow}& \sqrt{2 \gamma k_B \tau} e^{-\gamma t} \int_{0}^{t} e^{\gamma s} dW_s \qquad \text{almost surely}
\end{eqnarray}
and
\begin{eqnarray}\nonumber
\sigma_t &=& \sqrt{2 \gamma k_B \tau} \frac{1}{\left( 1- e^{-2 \gamma T} \right)^{1/2}} \\ \nonumber
&\underset{T \nearrow \infty}{\longrightarrow}& \sqrt{2 \gamma k_B \tau},
\end{eqnarray}
where the convergence is, obviously, uniform in $t$. So we have recovered the classical results after assuming that, not only the approach
to equilibrium, but also the dynamics once there, are of Ornstein-Uhlenbeck form; also, that the equilibrium is governed by Maxwell-Boltzmann
statistics and happens in the distant future.

\section{Energy, power, and forward stochastic differential equations}
\label{fsdes}

Let $(\Omega,\mathcal{F},\{\mathcal{F}_t\}_{t \ge 0},\mathbb{P})$ be a filtered probability space completed with the $\mathbb{P}-$null sets
in which a Wiener Process $\{B_t\}_{t \ge 0}$ is defined; moreover assume
$\mathcal{F}_t \supset \sigma(\{B_s, 0 \le s \le t\})$.
Now let us consider the circuit equation
\begin{eqnarray}\label{oui}
L dI_t &=& -R I_t dt + V dB_t, \\ \nonumber
\left. I_t \right|_{t=0} &=& I_0,
\end{eqnarray}
for a rigid wire loop, where $L,R,V>0$ are, respectively, the self-inductance, the resistance,
and the amplitude of the thermal fluctuations. The initial condition $I_0 \in L^4(\Omega)$ is a $\mathcal{F}_0-$measurable random variable.
Under these conditions it is clear that this equation has a unique strong solution that is global in time~\cite{kuo,oksendal}.
The approach developed in~\cite{gillespie} consists in solving this
forward stochastic differential equation in terms of the Ornstein-Uhlenbeck process, and then imposing
that the energy of the circuit, which is given by
$$
E_t = \frac12 L I_t^2,
$$
obeys the equipartition theorem in the long time limit. Since, in this case, the fluctuation-dissipation relation
could be seen as a consequence of the interplay of Langevin dynamics and equipartition of energy, it seems natural
to work directly with the energy instead of with the current.
We can use stochastic calculus to find the forward stochastic differential equation that $E_t$ obeys, but at this moment we
have to confront an old dilemma.
On one hand, we can select It\^o calculus~\cite{ito1,ito2} to find
\begin{equation}\label{ito}
d E_t = \frac{V^2}{2L} dt -2 \frac{R}{L} \, E_t \, dt + \sqrt{2 \frac{V^2}{L} \, E_t} \, dB_t;
\end{equation}
on the other hand, if the selection is Stratonovich calculus~\cite{stratonovich} we get
\begin{equation}\label{strat}
d E_t = -2 \frac{R}{L} \, E_t \, dt + \sqrt{2 \frac{V^2}{L} \, E_t} \circ dB_t.
\end{equation}
Actually, this double approach was studied before in the context of the random dispersal of the Brownian particle.
In~\cite{west} it is argued that the Stratonovich approach is superior to the It\^o approach. In~\cite{kampen},
however, the equality of both approaches from the methodological viewpoint is discussed, but the preference for the
Stratonovich one in those physical systems in which the fluctuations are external, such as in the cases studied in the present work,
is concluded. The results of this article were supported in~\cite{mmcc}, where the preference towards the Stratonovich interpretation
in continuous physical systems, again such as the ones studied herein, is mentioned. On the contrary, in~\cite{escudero}, the simpler
character of the It\^o interpretation in precisely this problem is defended. We shall further elaborate those arguments from now on.

Let us start with equation~\eqref{ito}, which admits the explicit solution
\begin{equation}\label{energy}
E_t = \frac12 L \left[ e^{-(R/L) t} \sqrt{\frac{2 E_0}{L}} + \frac{V}{L} \, \int_0^t e^{(R/L)(s-t)} \, dB_s \right]^2,
\end{equation}
a fact that is a direct consequence of the It\^o stochastic calculus rules, indeed
\begin{equation}\nonumber
dE_t= L I_t dI_t + \frac12 L (dI_t)^2,
\end{equation}
where
$$
I_t= e^{-(R/L) t} \sqrt{\frac{2 E_0}{L}} + \frac{V}{L} \, \int_0^t e^{(R/L)(s-t)} \, dB_s
$$
is the solution to equation~\eqref{oui}.
Moreover the solution is unique by the Wanatabe-Yamada theorem~\cite{wy} as was claimed in~\cite{escudero}.
Also, it is easy to see that the expected energy obeys the ordinary differential equation
\begin{equation}\nonumber
\frac{d \mathbb{E}(E_t)}{dt} = \frac{V^2}{2L} -2 \frac{R}{L} \, \mathbb{E}(E_t),
\end{equation}
which can be solved to yield
\begin{eqnarray}\nonumber
\mathbb{E}(E_t) &=& \frac{V^2}{4R} + \left[ \mathbb{E}(E_0) - \frac{V^2}{4R} \right] \exp\left(-2 \frac{R}{L}t\right) \\ \nonumber
&\underset{t \nearrow \infty}{\longrightarrow}& \frac{V^2}{4R}.
\end{eqnarray}
Now we apply the equipartition theorem
$$
\frac{V^2}{4R} =\lim_{t \nearrow \infty} \mathbb{E}(E_t)= \frac12 k_B \tau,
$$
so we conclude
$$
V= \sqrt{2 k_B \tau R},
$$
in perfect agreement with the previous results that were derived using the current rather than the energy~\cite{gillespie}.

On the other hand, the Watanabe-Yamada theorem is not applicable to Stratonovich stochastic differential equations~\cite{ce}.
As a consequence of this, equation~\eqref{strat} possesses an infinite number of solutions apart from~\eqref{energy}~\cite{escudero},
among which let us
first consider the family
\begin{eqnarray}\label{upet}
\overline{E}_t &=& \frac12 L \left[ e^{-(R/L) t} \sqrt{\frac{2 E_0}{L}} + \frac{V}{L} \, \int_0^t e^{(R/L)(s-t)} \, dB_s \right]^2
\, \mathlarger{\mathlarger{\mathbbm{1}}}_{t < T_1(\omega)}
\\ \nonumber & & +
\sum_{n=1}^{N} \frac{V^2}{2L} \left[ \int_{\mu_n+T_n(\omega)}^{t} e^{(R/L)(s-t)} \, dB_s \right]^2 \,
\mathlarger{\mathlarger{\mathbbm{1}}}_{\mu_n + T_n(\omega)<t<T_{n+1}(\omega)}
\\ \nonumber
& & + \frac{V^2}{2L} \left[ \int_{\mu_{N+1}+T_{N+1}(\omega)}^{t} e^{(R/L)(s-t)} \, dB_s \right]^2 \,
\mathlarger{\mathlarger{\mathbbm{1}}}_{t>\mu_{N+1}+T_{N+1}(\omega)},
\end{eqnarray}
for any set $\{\mu_n\}_{n=1}^{N+1}$, $N=1,2,\cdots$, of almost surely non-negative,
$L^0(\Omega)$, and $\mathcal{F}_{T_n(\omega)}-$measurable
random variables, where
$$
T_n := \inf \{ t>T_{n-1} + \mu_{n-1} + \lambda_{n-1} : \overline{E}_t=0 \}, \quad n=1,2,\cdots,
$$
with $T_0:=0=:\mu_0$, and
where $\{\lambda_n\}_{n=0}^N$ is an arbitrary set of almost surely positive,
$L^0(\Omega)$, and $\mathcal{F}_{T_{n}(\omega)}-$measurable random variables.
Note that all of these solutions are new and were not previously reported in~\cite{escudero}.
Now we select a subclass of this family of solutions and analyze its behavior for long times.

\begin{theorem}\label{lemcim}
Let $\overline{E}_t$ be as in~\eqref{upet} and assume $\{\lambda_n\}_{n=0}^N$ and $\{\mu_n\}_{n=1}^{N+1}$ are finite almost surely. Then
$$
\mathbb{E}(\overline{E}_t) \underset{t \nearrow \infty}{\longrightarrow} \frac{V^2}{4R}.
$$
\end{theorem}

\begin{proof}
First note that
\begin{eqnarray}\nonumber
&& \overline{E}_t -\frac{V^2}{2L} \left[ \int_{\mu_{N+1}+T_{N+1}(\omega)}^{t} e^{(R/L)(s-t)} \, dB_s \right]^2 \,
\mathlarger{\mathlarger{\mathbbm{1}}}_{t>\mu_{N+1}+T_{N+1}(\omega)}
\\ \nonumber
&=& \frac12 L \left[ e^{-(R/L) t} \sqrt{\frac{2 E_0}{L}} + \frac{V}{L} \, \int_0^t e^{(R/L)(s-t)} \, dB_s \right]^2
\, \mathlarger{\mathlarger{\mathbbm{1}}}_{t < T_1(\omega)}
\\ \nonumber & & +
\sum_{n=1}^{N} \frac{V^2}{2L} \left[ \int_{\mu_n+T_n(\omega)}^{t} e^{(R/L)(s-t)} \, dB_s \right]^2 \,
\mathlarger{\mathlarger{\mathbbm{1}}}_{\mu_n + T_n(\omega)<t<T_{n+1}(\omega)}
\\ \nonumber \\ \nonumber
&& \underset{t \nearrow \infty}{\longrightarrow} 0 \qquad \text{almost surely},
\end{eqnarray}
since all the $T_i$'s, $i=1,2,\cdots$, are finite almost surely~\cite{escudero}.
We also have the estimate
\begin{eqnarray}\nonumber
&& \left| \overline{E}_t - \frac{V^2}{2L} \left[ \int_{\mu_{N+1}+T_{N+1}(\omega)}^{t} e^{(R/L)(s-t)} \, dB_s \right]^2 \,
\mathlarger{\mathlarger{\mathbbm{1}}}_{t>\mu_{N+1}+T_{N+1}(\omega)} \right|
\\ \nonumber
&\le& 2 E_0  e^{-2(R/L) t} + \frac{V^2}{L} \left[ \int_0^t e^{(R/L)(s-t)} \, dB_s \right]^2
\\ \nonumber & & +
\sum_{n=1}^{N} \frac{V^2}{2L} \left[ \int_{\mu_n+T_n(\omega)}^{t} e^{(R/L)(s-t)} \, dB_s \right]^2
\mathlarger{\mathlarger{\mathbbm{1}}}_{t>\mu_n + T_n(\omega)},
\end{eqnarray}
and therefore
\begin{eqnarray}\nonumber
&& \mathbb{E}\left\{ \left| \overline{E}_t - \frac{V^2}{2L} \left[ \int_{\mu_{N+1}+T_{N+1}(\omega)}^{t} e^{(R/L)(s-t)} \, dB_s \right]^2 \,
\mathlarger{\mathlarger{\mathbbm{1}}}_{t>\mu_{N+1}+T_{N+1}(\omega)} \right| \right\}
\\ \nonumber
&\le& 2 \mathbb{E}\{E_0\} e^{-2(R/L) t} + \frac{V^2}{L}
\mathbb{E}\left\{\left[ \int_0^t e^{(R/L)(s-t)} \, dB_s \right]^2 \right\}
\\ \nonumber & & +
\sum_{n=1}^{N} \frac{V^2}{2L} \mathbb{E}\left\{\left[ \int_{\mu_n+T_n(\omega)}^{t} e^{(R/L)(s-t)} \, dB_s
\right]^2 \mathlarger{\mathlarger{\mathbbm{1}}}_{t>\mu_n + T_n(\omega)} \right\}
\\ \nonumber
&\le& 2 \mathbb{E}\{E_0\} e^{-2(R/L) t} + \frac{V^2}{2R} \left( 1-e^{-2Rt/L} \right) +
\sum_{n=1}^{N} \frac{V^2}{4R}
\\ \nonumber
&\le& 2 \mathbb{E}\{E_0\} + \frac{V^2}{2R} \left(1 + \frac{N}{2} \right)
\\ \nonumber &<& \infty,
\end{eqnarray}
where the last bound is of course uniform in $t$, and where we have used the It\^o isometry along with the string of equalities
and final inequality
\begin{eqnarray}\nonumber
&& \mathbb{E}\left\{\left[ \int_{\mu_n+T_n(\omega)}^{t} e^{(R/L)(s-t)} \, dB_s \right]^2
\mathlarger{\mathlarger{\mathbbm{1}}}_{t>\mu_n + T_n(\omega)} \right\}
\\ \nonumber
&=& \mathbb{E}\left\{\mathbb{E}\left\{\left[ \int_{\mu_n+T_n(\omega)}^{t} e^{(R/L)(s-t)} \, dB_s
\right]^2\Bigg|\mathcal{F}_{\mu_n+T_n(\omega)}\right\} \mathlarger{\mathlarger{\mathbbm{1}}}_{t>\mu_n + T_n(\omega)} \right\}
\\ \nonumber
&=& \mathbb{E}\left\{\mathbb{E}\left\{\int_{\mu_n+T_n(\omega)}^{t} e^{2(R/L)(s-t)} \, ds\Bigg|\mathcal{F}_{\mu_n+T_n(\omega)}\right\}
\mathlarger{\mathlarger{\mathbbm{1}}}_{t>\mu_n + T_n(\omega)} \right\}
\\ \nonumber
&=& \frac{L}{2R} \, \mathbb{E}\left\{ \left(1-\exp\left[ \frac{2R}{L}\left(\mu_n + T_n(\omega)-t\right) \right] \right)
\mathlarger{\mathlarger{\mathbbm{1}}}_{t>\mu_n + T_n(\omega)} \right\}
\\ \nonumber
&\le& \frac{L}{2R},
\end{eqnarray}
which follow from the tower property of conditional expectation
and the It\^o isometry. Now consider $\mathcal{A} \in \mathcal{F}$ with $\mathbb{P}\{\mathcal{A}\}=\delta$.
We have the estimate
\begin{eqnarray}\nonumber
&& \mathbb{E}\left\{ \left| \overline{E}_t - \frac{V^2}{2L} \left[ \int_{\mu_{N+1}+T_{N+1}(\omega)}^{t} e^{(R/L)(s-t)} \, dB_s \right]^2 \,
\mathlarger{\mathlarger{\mathbbm{1}}}_{t>\mu_{N+1}+T_{N+1}(\omega)} \right| \mathlarger{\mathlarger{\mathbbm{1}}}_{\mathcal{A}} \right\}
\\ \nonumber
&\le& 2 \mathbb{E}\left\{E_0 \mathlarger{\mathlarger{\mathbbm{1}}}_{\mathcal{A}}\right\} e^{-2(R/L) t}
+ \frac{V^2}{L} \mathbb{E}\left\{\left[ \int_0^t e^{(R/L)(s-t)} \, dB_s \right]^2 \mathlarger{\mathlarger{\mathbbm{1}}}_{\mathcal{A}} \right\}
\\ \nonumber & & +
\sum_{n=1}^{N} \frac{V^2}{2L} \mathbb{E}\left\{\left[ \int_{\mu_n+T_n(\omega)}^{t} e^{(R/L)(s-t)} \, dB_s \right]^2
\mathlarger{\mathlarger{\mathbbm{1}}}_{t>\mu_n + T_n(\omega)} \, \mathlarger{\mathlarger{\mathbbm{1}}}_{\mathcal{A}} \right\}
\\ \nonumber
&\le& 2 \mathbb{E}\left\{E_0^2\right\}^{1/2} \mathbb{P}(\mathcal{A})^{1/2}
+ \frac{V^2}{L} \mathbb{E}\left\{\left[ \int_0^t e^{(R/L)(s-t)} \, dB_s \right]^4 \right\}^{1/2}
\mathbb{P}(\mathcal{A})^{1/2}
\\ \nonumber & & +
\sum_{n=1}^{N} \frac{V^2}{2L} \mathbb{E}\left\{\left[ \int_{\mu_n+T_n(\omega)}^{t} e^{(R/L)(s-t)} \, dB_s \right]^4
\mathlarger{\mathlarger{\mathbbm{1}}}_{t>\mu_n + T_n(\omega)} \right\}\mathbb{P}(\mathcal{A})^{1/2}
\\ \nonumber
&\le& 2 \mathbb{E}\left\{E_0^2\right\}^{1/2} \delta^{1/2}
+ \frac{\sqrt{3}V^2}{2R} \delta^{1/2}
\\ \nonumber & & +
\sum_{n=1}^{N} \frac{V^2}{2L}
\mathbb{E}\left\{\mathbb{E}\left\{\left[ \int_{\mu_n+T_n(\omega)}^{t} e^{(R/L)(s-t)} \, dB_s \right]^4\Bigg|\mathcal{F}_{\mu_n+T_n(\omega)}
\right\} \mathlarger{\mathlarger{\mathbbm{1}}}_{t>\mu_n + T_n(\omega)} \right\}\delta^{1/2}
\\ \nonumber
&\le& 2 \mathbb{E}\left\{E_0^2\right\}^{1/2} \delta^{1/2}
+ \frac{\sqrt{3}V^2}{2R} \delta^{1/2} +
\sum_{n=1}^{N} \frac{\sqrt{3} V^2}{4R} \delta^{1/2}
\\ \nonumber
&=& \left[ 2 \mathbb{E}\left\{E_0^2\right\}^{1/2}
+ \frac{\sqrt{3}V^2}{2R} \left( 1 + \frac{N}{2} \right) \right] \delta^{1/2},
\end{eqnarray}
which follows from the H\"older inequality
\begin{eqnarray}\nonumber
\mathbb{E}\left\{E_0 \mathlarger{\mathlarger{\mathbbm{1}}}_{\mathcal{A}}\right\}
&\le& \mathbb{E}\left\{E_0^2\right\}^{1/2} \mathbb{E}\left\{\mathlarger{\mathlarger{\mathbbm{1}}}^2_{\mathcal{A}}\right\}^{1/2} \\ \nonumber
&=& \mathbb{E}\left\{E_0^2\right\}^{1/2} \mathbb{E}\left\{\mathlarger{\mathlarger{\mathbbm{1}}}_{\mathcal{A}}\right\}^{1/2} \\ \nonumber
&=& \mathbb{E}\left\{E_0^2\right\}^{1/2} \mathbb{P}(\mathcal{A})^{1/2},
\end{eqnarray}
and analogous H\"older inequalities for the other terms, the tower property, and the inequalities
\begin{eqnarray}\nonumber
\mathbb{E}\left\{\left[ \int_0^t e^{(R/L)(s-t)} \, dB_s \right]^4 \right\} &=& 3\frac{L^2}{4R^2} \left( 1-e^{-2Rt/L} \right)^2
\\ \nonumber
&\le& \frac{3L^2}{4R^2}
\end{eqnarray}
and
\begin{eqnarray}\nonumber
&& \mathbb{E}\left\{\left[ \int_{\mu_n+T_n(\omega)}^{t} e^{(R/L)(s-t)} \, dB_s \right]^4 \Bigg|\mathcal{F}_{\mu_n+T_n(\omega)}
\right\} \mathlarger{\mathlarger{\mathbbm{1}}}_{t>\mu_n + T_n(\omega)}
\\ \nonumber
&=& 3\frac{L^2}{4R^2} \left\{ 1-\exp\left[ \frac{2R}{L}\left(\mu_n + T_n(\omega)-t\right) \right] \right\}^2
\mathlarger{\mathlarger{\mathbbm{1}}}_{t>\mu_n + T_n(\omega)}
\\ \nonumber
&\le& \frac{3L^2}{4R^2},
\end{eqnarray}
since respectively
$$
\int_0^t e^{(R/L)(s-t)} \, dB_s \sim \mathcal{N}\left(0,\frac{L}{2R} \left( 1-e^{-2Rt/L} \right)\right)
$$
and
$$
\int_{\mu_n+T_n(\omega)}^{t} e^{(R/L)(s-t)} \, dB_s \Bigg|\mathcal{F}_{\mu_n+T_n(\omega)} \sim
\mathcal{N}\left(0,\frac{L}{2R} \left\{ 1-\exp\left[ \frac{2R}{L}\left(\mu_n + T_n(\omega)-t\right) \right] \right\}\right).
$$
Then consequently
\begin{equation}\nonumber
\mathbb{E}\left\{ \left| \overline{E}_t - \frac{V^2}{2L} \left[ \int_{\mu_{N+1}+T_{N+1}(\omega)}^{t} e^{(R/L)(s-t)} \, dB_s \right]^2 \,
\mathlarger{\mathlarger{\mathbbm{1}}}_{t>\mu_{N+1}+T_{N+1}(\omega)} \right| \mathlarger{\mathlarger{\mathbbm{1}}}_{\mathcal{A}} \right\}
\le \varepsilon
\end{equation}
for all $\varepsilon >0$, after choosing
$$
\delta = \frac{\varepsilon^2}{\left[ 2 \mathbb{E}\left\{E_0^2\right\}^{1/2}
+ \dfrac{\sqrt{3}V^2}{2R} \left( 1 + \dfrac{N}{2} \right) \right]^2},
$$
and by the Vitali convergence theorem~\cite{folland}
\begin{eqnarray}\nonumber
&& \lim_{t \nearrow \infty} \mathbb{E}\left\{ \left| \overline{E}_t - \frac{V^2}{2L} \left[ \int_{\mu_{N+1}+ T_{N+1}(\omega)}^{t}
e^{(R/L)(s-t)} \, dB_s \right]^2 \,
\mathlarger{\mathlarger{\mathbbm{1}}}_{t>\mu_{N+1}+ T_{N+1}(\omega)} \right| \right\} \\ \nonumber
&=& \mathbb{E}\left\{ \lim_{t \nearrow \infty} \left| \overline{E}_t - \frac{V^2}{2L} \left[
\int_{\mu_{N+1}+ T_{N+1}(\omega)}^{t} e^{(R/L)(s-t)} \, dB_s \right]^2 \,
\mathlarger{\mathlarger{\mathbbm{1}}}_{t>\mu_{N+1}+ T_{N+1}(\omega)} \right| \right\} \\ \nonumber
&=& 0.
\end{eqnarray}
Therefore
\begin{eqnarray}\nonumber
&& \lim_{t \nearrow \infty} \mathbb{E}\left\{ \overline{E}_t \right\} \\ \nonumber
&=& \lim_{t \nearrow \infty}
\mathbb{E} \left\{ \frac{V^2}{2L} \left[ \int_{\mu_{N+1}+ T_{N+1}(\omega)}^{t} e^{(R/L)(s-t)} \, dB_s \right]^2 \,
\mathlarger{\mathlarger{\mathbbm{1}}}_{t>\mu_{N+1}+ T_{N+1}(\omega)} \right\} \\ \nonumber
&=& \lim_{t \nearrow \infty}
\mathbb{E} \left\{ \mathbb{E} \left\{ \frac{V^2}{2L} \left[ \int_{\mu_{N+1}+ T_{N+1}(\omega)}^{t} e^{(R/L)(s-t)} \, dB_s \right]^2 \,
\mathlarger{\mathlarger{\mathbbm{1}}}_{t>\mu_{N+1}+ T_{N+1}(\omega)} \Bigg| \mathcal{F}_{\mu_{N+1}+ T_{N+1}} \right\}\right\} \\ \nonumber
&=& \lim_{t \nearrow \infty}
\mathbb{E} \left\{ \mathbb{E} \left\{ \frac{V^2}{2L} \int_{\mu_{N+1}+ T_{N+1}(\omega)}^{t} e^{2(R/L)(s-t)} \, ds \,
\Bigg| \mathcal{F}_{\mu_{N+1}+ T_{N+1}} \right\}
\mathlarger{\mathlarger{\mathbbm{1}}}_{t>\mu_{N+1}+ T_{N+1}(\omega)}\right\} \\ \nonumber
&=& \lim_{t \nearrow \infty}
\mathbb{E} \left\{ \frac{V^2}{4R} \left\{ 1-\exp\left[ \frac{2R}{L}\left(\mu_n + T_n(\omega)-t\right) \right] \right\}
\mathlarger{\mathlarger{\mathbbm{1}}}_{t>\mu_{N+1}+ T_{N+1}(\omega)}\right\} \\ \nonumber
&=& \frac{V^2}{4R} \,
\mathbb{E} \left\{ \lim_{t \nearrow \infty} \left\{ 1-\exp\left[ \frac{2R}{L}\left(\mu_n + T_n(\omega)-t\right) \right] \right\}
\mathlarger{\mathlarger{\mathbbm{1}}}_{t>\mu_{N+1}+ T_{N+1}(\omega)}\right\} \\ \nonumber
&=& \frac{V^2}{4R},
\end{eqnarray}
by the tower property, the It\^o isometry, and the dominated convergence theorem.
\end{proof}
Note that all of these solutions are not physical, as they imply a zero energy, and therefore zero current, during some intervals
of time, despite the presence of non-vanishing thermal fluctuations.
But, in spite of their unphysical character, their long time behavior is still compatible with the equipartition theorem provided we assume
$V= \sqrt{2 k_B \tau R}$, which is nothing but the fluctuation-dissipation relation.

Let us now move to a second family of solutions to equation~\eqref{strat}, in particular to the one given by the explicit formula
\begin{eqnarray}\nonumber
\underline{E}_t &=& \frac12 L \left[ e^{-(R/L) t} \sqrt{\frac{2 E_0}{L}} + \frac{V}{L} \, \int_0^t e^{(R/L)(s-t)} \, dB_s \right]^2
\, \mathlarger{\mathlarger{\mathbbm{1}}}_{t < T_1'(\omega)}
\\ \nonumber & & +
\sum_{n=1}^{N} \frac{V^2}{2L} \left[ \int_{\mu_n' + T_n'(\omega)}^{t} e^{(R/L)(s-t)} \, dB_s \right]^2 \,
\mathlarger{\mathlarger{\mathbbm{1}}}_{\mu_n' + T_n'(\omega)<t<T_{n+1}'(\omega)},
\end{eqnarray}
for any set $\{\mu_n'\}_{n=1}^N$, $N \in \mathbb{N} \cup \{\infty\}$, of almost surely non-negative,
$L^0(\Omega)$, and $\mathcal{F}_{T_n'(\omega)}-$measurable
random variables, where
$$
T_n' := \inf \{ t>T_{n-1}' + \mu_{n-1}' + \lambda_{n-1}' : \underline{E}_t=0 \}, \quad n=1,2,\cdots,
$$
with $T_0':=0=:\mu_0'$, and
where $\{\lambda_n'\}_{n=0}^N$ is an arbitrary set of almost surely positive,
$L^0(\Omega)$, and $\mathcal{F}_{T_{n}'(\omega)}-$measurable random variables.
Note that some of these solutions are also new and were not reported in~\cite{escudero} (precisely, all the cases with $E_0=0$),
but they have been built using the same philosophy.

To carry out the long time analysis assume moreover that
$N < \infty$ and $\{\lambda_n'\}_{n=0}^N$ and $\{\mu_n'\}_{n=1}^{N+1}$ are finite almost surely.
Because every stopping time $T_i'$, with $i=1,2,\cdots$, is finite almost surely~\cite{escudero},
we deduce the long time behavior
$$
\underline{E}_t \underset{t \nearrow \infty}{\longrightarrow} 0 \qquad \text{almost surely}
$$
for whatever absolute temperature $\tau$; consequently these solutions lead to zero energy, and consequently zero current,
in the long time with probability one, despite of the presence of non-vanishing thermal fluctuations.
Thus we conclude that these solutions are unphysical, and also that both the equipartition theorem and
fluctuation-dissipation relation are meaningless for them.

To finish this section we would like to emphasize that these results are not a consequence of strictly considering the energy of the circuit,
but could also be reached from the dynamics of other physical quantities.
The power dissipated the circuit is given by
$$
D_t = R I_t^2
$$
and, as we did before for the energy, we can use stochastic calculus to find the forward stochastic differential equation it obeys.
If we select It\^o calculus we find
\begin{equation}\nonumber
d D_t = \frac{R V^2}{L^2} dt -2 \frac{R}{L} \, D_t \, dt + 2 \sqrt{\frac{R V^2}{L^2} \, D_t} \, dB_t;
\end{equation}
while if the selection is Stratonovich calculus we get
\begin{equation}\nonumber
d D_t = -2 \frac{R}{L} \, D_t \, dt + 2 \sqrt{\frac{R V^2}{L^2} \, D_t} \circ dB_t.
\end{equation}
On the other hand, the fluctuation-dissipation relation implies
\begin{eqnarray}\nonumber
\lim_{t \nearrow \infty} \mathbb{E}(D_t) &=& \lim_{t \nearrow \infty} \frac{2R}{L} \mathbb{E}(E_t) \\ \nonumber
&=& \frac{k_B \tau R}{L},
\end{eqnarray}
a result that is achievable from the unique solution of the It\^o equation, just as we did before in the case of the energy.
Also note that arguing as in Theorem~\ref{lemcim} we can find infinitely many solutions to the Stratonovich forward stochastic differential equation that fulfil the convergence
$$
\mathbb{E}(\overline{D}_t) \underset{t \nearrow \infty}{\longrightarrow} \frac{R V^2}{L^2},
$$
and therefore the equipartition theorem whenever the fluctuation-dissipation relation $V=\sqrt{2 k_B \tau R}$ holds.
But we can find as well infinitely many solutions that fulfil
$$
\underline{D}_t \underset{t \nearrow \infty}{\longrightarrow} 0 \qquad \text{almost surely}
$$
for any triplet of positive parameters $\{L,V,R\}$, and for any absolute temperature $\tau$,
and therefore neither the equipartition theorem nor the fluctuation-dissipation relation make sense for them.

\section{Ideal circuit and fluctuating thermal amplitude}
\label{ideal}

The aim of this section is to show how the situation in the previous one is exactly the same as the most
basic formulation of the It\^o vs Stratonovich dilemma, that is, selecting the precise meaning of
a formal basic stochastic integral.
For this we consider a formal Gaussian white noise process $\xi_t$ and an associated formal circuit equation
\begin{eqnarray}\label{bbt}
L \frac{dI_t}{dt} &=& V B_t \xi_t, \\ \nonumber
\left. I_t \right|_{t=0} &=& I_0,
\end{eqnarray}
which models a fluctuating amplitude of the thermal fluctuations in an ideal circuit with no resistance.
If we consider the white noise process to be the formal derivative of the Brownian motion, i.~e.
\begin{equation}\nonumber
\xi_t= \frac{\eth B_t}{dt},
\end{equation}
we can rewrite equation~\eqref{bbt} as
\begin{eqnarray}\nonumber
L dI_t &=& V B_t \eth B_t, \\ \nonumber
\left. I_t \right|_{t=0} &=& I_0,
\end{eqnarray}
where $\eth$ denotes a formal stochastic integration scheme. This equation can be formally solved to find
$$
I_t = I_0 + \frac{L}{V} \int_{0}^{t} B_s \eth B_s,
$$
where the last integral is still to be defined.
If we defined this integral in the sense of Riemann-Stieltjes we would find
\begin{equation}\nonumber
\int_{0}^{t} B_s \eth B_s = \mathbb{P}-\lim_{\left| \Pi_n \right| \to 0} \sum_{i=1}^n B_{t_{i-1}^*} \left(B_{t_i}-B_{t_{i-1}} \right),
\end{equation}
where $\Pi = \lbrace 0 =: t_0, t_1, t_2, ..., t_n := t \rbrace $ is a partition of the interval $[0,t]$, $\left| \Pi_n \right| = \max _{1 \leq i \leq n} \left(t_i - t_{i-1}\right)$, and $t_i^* \in [t_{i-1},t_i]$ is given by the convex linear combination $t_i^* = \alpha t_{i} + (1-\alpha) t_{i-1}$
with $\alpha \in [0,1]$ arbitrary. Of course, if the integral existed in the Riemann-Stieltjes sense the result would be independent of $\alpha$.
However, it is well known that the result is $\alpha-$dependent and reads~\cite{evans}
\begin{equation}\nonumber
\int_{0}^{t} B_s \eth B_s = \frac{B_t^2}{2} + \left( \alpha -\frac12 \right) t.
\end{equation}
Therefore the solution to equation~\eqref{bbt} is
$$
I_t = I_0 + \frac{L}{V} \left[ \frac{B_t^2}{2} + \left( \alpha -\frac12 \right) t \right], \qquad \alpha \in [0,1];
$$
consequently it is a multivalued stochastic process. However, causality in physics imposes the solution to be an single-valued
stochastic process, and thus, a unique process has to be selected from this one-parameter family of solutions. Of course,
this selection has to be done solely on physical grounds. Since the nature of the fluctuations is completely random, the induced
current should be isotropic on average, and therefore we should have $\mathbb{E}(I_t)=0$ for all $t \ge 0$. This imposes $\alpha=0$ or,
in other words, the It\^o interpretation of noise.

This is a very basic formulation of the It\^o vs Stratonovich dilemma. According to van Kampen, \eqref{bbt} is not actually an equation
but a \emph{pre-equation}, and it only becomes an actual equation when a noise interpretation is given~\cite{kampen}. However, the only
fundamental difficulty associated to~\eqref{bbt} is its infinite multiplicity of solutions. Adding an interpretation of noise is equivalent to
choosing one solution out of this set; something that can only be done using physical arguments. The same reasoning is applicable to the majority
of the It\^o vs Stratonovich dilemmas one finds in the physical literature.
The situation is, in essence, identical to that of equation~\eqref{strat}: it admits infinitely many solutions and only one has to be chosen, with
this selection strictly based on physical grounds. In other words, and using the terminology of van Kampen, adding an interpretation of noise
does not necessarily transform a \emph{pre-equation} into an equation. This could be so in many situations, but there are known counterexamples both
in the Stratonovich case (such as in section~\ref{fsdes}, \cite{ce}, and~\cite{escudero}) and in the It\^o one (see for instance~\cite{ander}).
Finally, it is convenient to emphasize that there is a right unique solution to problems~\eqref{strat} and~\eqref{bbt} only if they are
intended to describe a physical phenomenon and not simply considered as abstract mathematical models. Equivalently, there is nothing mathematically
wrong about their respective infinite solution sets: all these solutions are equally admissible. Moreover, if one changed the physical significances
of these models, the right physical solutions would in general differ from those highlighted herein.

\section{Conclusions}
\label{conclusions}

In this work we have examined the classical fluctuation-dissipation relation under the light of stochastic analytical
methods that have not been employed before, to the best of our knowledge, in this problem. We have started considering
the random dispersal of the Brownian particle subjected to a thermal bath. Instead studying the Langevin equation
posed forward in time, a classical theoretical approach to this problem, we have considered this equation posed backwards
in time. This is a natural approach to the problem, since as the fluctuation-dissipation relation arises as a consequence of
Maxwell-Boltzmann statistics in the long time limit, we can impose this statistics to the terminal condition. Moreover,
since a single backward stochastic differential equation has to be solved for two stochastic processes in order to keep the
adaptability of the solution, this enables us to solve this equation simultaneously for the velocity of the Brownian particle
and the strength of the thermal fluctuations, so any relation among them can be studied in detail. Our results indicate that
the classical fluctuation-dissipation relation is only recovered after assuming Ornstein-Uhlenbeck-type dynamics at equilibrium
and the occurrence of the terminal time in the distant future. Under these assumptions we found in section~\ref{bsdes} that
\begin{eqnarray}\nonumber
V_t &=& \sqrt{2 \gamma k_B \tau} e^{-\gamma t} \int_{0}^{t} e^{\gamma s} dW_s, \\ \nonumber
\sigma_t &=& \sqrt{2 \gamma k_B \tau},
\end{eqnarray}
which is what one would classically expect. Note however that we have also obtained the additional pairs of results \eqref{v1}--\eqref{s1},
\eqref{v2}--\eqref{s2}, and \eqref{v3}--\eqref{s3}, which generalize this one when the second or both of these two assumptions
are removed. To the best of our knowledge these pairs of results are new, and in our opinion illustrate the role of these assumptions in a more
transparent way than other approaches. We regard them as some of the advantages of the use of backward rather
than forward stochastic differential equations in this problem.

Subsequently we have studied the fluctuation-dissipation relation in the context of Johnson noise in electric circuits. In this case,
instead of relying on the more traditional approach of studying the Langevin equation for the electric current, we have derived
forward stochastic differential equations for the energy of the circuit. If such an equation is derived using It\^o calculus,
this leads to a unique solution that perfectly reproduces the classical results. If alternatively Stratonovich calculus is used,
we end up with infinitely many solutions, of which only one is physical, precisely the solution derived using It\^o calculus.
Despite of the spurious character of these solutions, some them (in fact, infinitely many) still obey equipartition of energy and
the fluctuation-dissipation relation in the long time limit.
While for others (in fact, infinitely many too), both of these classical results are meaningless. The same conclusions arise if
instead of the energy we focus on the power dissipated by the circuit.

Our present results are important in the light of the established consensus regarding the interpretation of noise.
In~\cite{rs} one finds this consensus summarized as ``computers are It\^o and circuits are Stratonovich'' in the field
of electric engineering. However, we have found that circuits can be It\^o too. A similar conclusion is reached in~\cite{smmcc}
where the authors claim that ``the Stratonovich results, however, accurately describe what actually happens in nature'' in contrast
to the It\^o ones. But however, sometimes the unique solution to an It\^o equation could be of physical nature, while the Stratonovich
equation possesses an infinite number of unphysical solutions. While in~\cite{moon} we can read ``in most areas of physical science (...)
Stratonovich calculus is preferred, mainly due to the consistency with the results emerging from the Fokker-Planck equation and the
ﬂuctuation-dissipation theorem''; nevertheless we have found solutions to a Stratonovich stochastic differential equation which
are totally inconsistent with the fluctuation-dissipation relation, unlike the It\^o solution. In summary, we agree with the conclusions
in~\cite{moon} in that the It\^o vs Stratonovich dilemma should be resolved on a case by case basis.
Thus, on different problems from the one studied herein, one should {\it a priori} consider the two classical interpretations of noise
along with other possible meanings associated to a stochastic differential equation model~\cite{yuanao}.
Our main conclusion from this section
is that general guidelines can be useful at times for the noise interpretation dilemma, but one should be aware of possible counterexamples too.
Moreover, the connection between these results and the most basic version of the It\^o vs Stratonovich dilemma was highlighted in section~\ref{ideal},
which shows that both essentially face the same question.

The problem of the interpretation of noise has some reflection also in the second section. A backward stochastic differential equation
cannot be interpreted in the sense of It\^o, as this would lead to the lack of adaptability, and therefore to the non-existence, of the solution.
The interpretation in the sense of Pardoux and Peng~\cite{pp} guarantees the adaptability of the solution, what is a physical requisite in
the problems at hand. However, at least mathematically speaking, such an equation is well posed,
even if its solution is not adapted, if interpreted according to some anticipating
stochastic integral. The problem of the interpretation of noise in the anticipating setting is summarized in~\cite{be,escudero2,erc}.
We leave as an open question the application of such techniques in systems in which the fluctuation-dissipation relation is of interest.

\section*{Acknowledgements}

This work has been partially supported by the Government of Spain (Ministerio de Ciencia, Innovaci\'on y Universidades)
through Project PGC2018-097704-B-I00.

\vskip5mm
\noindent
{\footnotesize
Carlos Escudero\par\noindent
Departamento de Matem\'aticas Fundamentales\par\noindent
Universidad Nacional de Educaci\'on a Distancia\par\noindent
{\tt cescudero@mat.uned.es}\par\vskip1mm\noindent
}
\end{document}